\documentclass[11pt,letter paper,twoside,openright,final]{article} 
\usepackage[english]{babel}
\usepackage{amsfonts}
\usepackage[utf8]{inputenc}
\usepackage{amssymb,amsmath,amsthm}
\usepackage{amsthm}
\usepackage{enumitem}
\usepackage[pdftex]{graphicx}
\usepackage{float}
\usepackage{verbatim}
\usepackage[dvipsnames]{xcolor}
\usepackage[pdftex]{hyperref}
\usepackage[inner=2.8cm, outer= 2.8cm, top=2.5cm, bottom=2.2cm]{geometry}
\hypersetup{colorlinks, citecolor=black, filecolor=black, linkcolor=black, urlcolor=black, pdftex}
\theoremstyle{definition}
\newtheorem{theorem}{Theorem}[section]
\newtheorem{lemma}[theorem]{Lemma}
\newtheorem{corollary}[theorem]{Corollary}
\newtheorem{construction}[theorem]{Construction}
\newtheorem{proposition}[theorem]{Proposition}

\def\arraystretch{1,5}

\def \ZZ {\mathbb {Z}}

\newcommand{\modulo}[3]{#1 \equiv #2 \: (\textrm{mod }#3)}
\newcommand{\modul}[2]{#1 \: (\textrm{mod }#2)}

\newcommand\blfootnote[1]{%
  \begingroup

  \renewcommand\thefootnote{}\footnote{#1}%

  \addtocounter{footnote}{-1}%

  \endgroup

}

\begin{document}

\begin{center}
\Large{\textbf{Tetravalent distance magic graphs of small order and an infinite family of examples}} \\ [+4ex]
Ksenija Rozman{\small$^{a,*}$}, Primo\v z \v Sparl{\small$^{a, b, c}$} \\ [+2ex]
{\it \small
$^a$Institute of Mathematics, Physics and Mechanics, Ljubljana, Slovenia\\
$^b$University of Ljubljana, Faculty of Education, Ljubljana, Slovenia\\
$^c$University of Primorska, Institute Andrej Maru\v si\v c, Koper, Slovenia\\
}
\end{center}


\blfootnote{

Email addresses:
ksenija.rozman@pef.uni-lj.si (Ksenija Rozman),
primoz.sparl@pef.uni-lj.si (Primo\v z \v Sparl)

* - corresponding author
}

\hrule

\begin{abstract}

A graph of order $n$ is {\em distance magic} if it admits a bijective labeling of its vertices with integers from $1$ to $n$ such that each vertex has the same sum of the labels of its neighbors.

This paper contributes to the long term project of characterizing all tetravalent distance magic graphs. With the help of a computer we find that out of almost nine million connected tetravalent graphs up to order $16$ only nine are distance magic. In fact, besides the six well known wreath graphs there are only three other examples, one of each of the orders $12$, $14$ and $16$. We introduce a generalization of wreath graphs, the so-called quasi wreath graphs, and classify all distance magic graphs among them. This way we obtain infinitely many new tetravalent distance magic graphs. Moreover, the two non-wreath graphs of orders $12$ and $14$ are quasi wreath graphs while the one of order $16$ can be obtained from a quasi wreath graph of order $14$ using a simple construction due to Ková\v{r}, Fronček and Ková\v{r}ová.

\end{abstract}

\hrule

\begin{quotation}


\end{quotation} 

\section{Introduction}

Graph labelings are assignments of labels (in most cases integers) to vertices and/or edges of graphs. They were first introduced in mid 1960s and have since attracted many mathematicians, as is witnessed by the wide variety of graph labeling variants studied and thousands of papers on the topic published (a list is compiled in the comprehensive survey by Gallian \cite{Gallian}). In this paper we focus on distance magic labelings which were first studied independently by Vilfred \cite{Vilfred} and by Miller, Rodger and Simanjuntak \cite{Miller} under the names sigma labelings and 1-vertex magic vertex labelings, respectively (see also \cite[Section 5.6]{Gallian}).

According to the definition which has been commonly used ever since its introduction in 2009 in \cite{SFM} (but see Section 2, where we give a somewhat different definition which appears to be more convenient for our purposes), a {\em distance magic labeling} of a graph of order $n$ is a bijective labeling of vertices with integers from 1 to $n$, such that the {\em weight} of each vertex $v$ (that is, the sum of the labels of the neighbors of $v$) is equal to a constant $\kappa$, called the {\em magic constant}. A graph is said to be {\em distance magic} if there exists a distance magic labeling of its vertices. It is well known and easy to see that in the case of regular graphs of valency $r$, the distance magic constant is $r(n+1)/2$, implying that there are no $r$-regular distance magic graphs with $r$ odd \cite{Miller}. While it is clear that the only connected 2-regular distance magic graph is the 4-cycle, the situation with 4-regular graphs seems to be much more intriguing. In fact, even though Rao proposed the problem of characterizing all 4-regular distance magic graphs already back in 2008 \cite{Rao}, we thus far only have a few partial results, for instance \cite{ACP, CF, KFK, MS, RSP, RS} (we do mention that there are several other results on distance magic graphs in general, see for instance \cite{Gallian, GodSin18, MS6}).

The first natural step in the investigation of the above mentioned problem is to determine the orders for which connected tetravalent distance magic graphs exist. That at least one example exists for each even order $n \ge 6$ is easy to see as one can simply take the wreath graph $W(n/2)$ of order $n$ (see Section \ref{sec:2} for the definition). For odd orders it was first shown in \cite{KFK} (using a computer) that there are no examples up to order 15 and then using a nice simple construction that at least one example exists for each odd order $n \ge 17$. 

One of the next natural steps is to investigate small tetravalent distance magic graphs and try to identify nice infinite families to which these graphs belong. Using a computer we verify (see Section \ref{small graphs}) that distance magic examples are extremely rare among all tetravalent graphs, at least for small orders. In fact, up to order 16 there are only nine connected distance magic tetravalent graphs, namely the wreath graphs $W(n)$, $3 \le n \le 8$, and three additional examples, one of each of the orders $12$, $14$ and $16$.

In Section \ref{small graphs} we introduce the family of so-called {\em quasi wreath graphs} (which contains the above non-wreath graphs of orders 12 and 14) and classify its distance magic members. The following is our main result (see Section \ref{small graphs} for definitions).

\begin{theorem}\label{t1}
The quasi wreath graph $QW(S)$ is distance magic if and only if it consists of any number of segments of length congruent to 3 modulo 4 and an even number of segments of length congruent to 1 modulo 4. 
\end{theorem}

The family of quasi wreath graphs thus provides infinitely many tetravalent distance magic graphs. In fact, as we explain in Section \ref{sec:future}, it provides at least one non-wreath distance magic example for each even order $n \ge 12$, except for $n =16$. Moreover, the only non-wreath example of order 16 can be obtained from the distance magic quasi wreath graph of order 14 using the simple construction from \cite{KFK}.

\section{Preliminaries} \label{sec:2}
In this paper we consider only finite, connected, simple and undirected graphs. First we give some basic notation that we use. For a graph $\Gamma$ we denote its vertex set by $V(\Gamma)$, its edge set by $E(\Gamma)$, and the neighborhood of a vertex $u \in V(\Gamma)$, which is the set of vertices adjacent to $u$, by $N(u)$. The ring of residue classes modulo $n$, where $n$ is a positive integer, is denoted by $\ZZ_n$. 

As mentioned in the Introduction we now give a slightly different definition of a distance magic labeling of a tetravalent graph, which however leads to the same definition of being distance magic. Let $\Gamma =(V, E)$ be a tetravalent graph of order $n$ and let $N = \{1-n, 3-n, \ldots, n-1\}$. Then a {\em distance magic labeling} of $\Gamma$ is a bijection $\ell \colon V \longrightarrow N$ such that the weight of each vertex is equal to $0$. 

Clearly, if $\ell$ is a distance magic labeling according to this definition, then setting $\tilde{\ell}(u)=(\ell(u)+n+1)/2$ for each vertex $u$ one gets a distance magic labeling according to \cite{SFM} and conversely, if $\tilde{\ell}$ is a distance magic labeling according to \cite{SFM}, then setting $\ell(u)=2 \tilde{\ell}(u)-1-n$ for each vertex $u$ one gets a distance magic labeling according to our definition. Throughout the rest of this paper we will thus always be working with our definition of a distance magic labeling. We point out that in this setting the weight of each vertex for such a labeling is 0. 

For $n \ge 3$, the {\em wreath graph} $W(n)$ is the tetravalent graph of order $2n$ with vertex set $V=\{u_i \colon i \in \ZZ_n\} \cup \{v_i \colon i \in \ZZ_n\}$ where for each $i \in \ZZ_n$ it holds that $N(u_i)=N(v_i)=\{u_{i\pm 1}, v_{i \pm 1}\}$, all indices computed modulo $n$. Clearly, the graphs $W(n)$, $n \ge 3$, are all distance magic, since one can simply label the pair $u_i, v_i$ by $2n-2i-1$ and $-(2n-2i-1)$ (but see also \cite[Theorem 1.1]{MS}). 

\section {The quasi wreath graphs } \label{small graphs}

In 1999 Meringer \cite{Meringer} developed an algorithm enabling construction of regular graphs of given valency up to a specified small enough order. The database of all connected tetravalent graphs of small orders is available at \cite{Meringer} (see also \cite{Houseofgraphs}). It turns out that there are $906\,331$ connected tetravalent graphs up to order 15. To analyze which of them are distance magic we used the following idea pointed out in \cite[Lemma 2.1]{MS} and its immediate corollary.

\begin{lemma}\cite{MS} \label{ev}
Let $\Gamma=(V,E)$ be a regular graph of even valency and order $n$. Then $\Gamma$ is distance magic if and only if 0 is an eigenvalue of $\Gamma$ and there exists an eigenvector for the eigenvalue 0 with the property that a certain permutation of its entries results in the arithmetic sequence $ 1-n, 3-n, 5-n, \ldots ,n-3, n-1. $
\end{lemma}

\begin{corollary} \label{cor}
Let $\Gamma=(V,E)$ be a regular graph of even valency and order $n$. Suppose that 0 is an eigenvalue of $\Gamma$ and let $B$ be the basis of the corresponding eigenspace. If there exist integers $i,j$ with $1 \le i < j \le n$ such that $b(i)=b(j)$ for all $b \in B$ then $\Gamma$ is not distance magic.
\end{corollary}

Using a computer one finds that out of the $906\,331$ connected tetravalent graphs up to order 15 all but 9 graphs are ruled out as candidates for distance magic graphs by Corollary~\ref{cor}. In particular, there are $1,1,1,2,2,2$ candidates of order $6, 8, 10, 12, 14, 15$, respectively. We already know that the 5 wreath graphs $W(n)$, $3 \le n \le 7$, are distance magic which leaves us with two candidates of order 15 and one candidate of each of the orders 12 and 14. For the two graphs of order 15 a computer check reveals that they do not satisfy the conditions of Lemma~\ref{ev} and are thus not distance magic (which is consistent with \cite{KFK}). On the other hand the remaining two graphs of orders 12 and 14 are indeed distance magic. They are presented in Figure \ref{pic:smallgraphs} where a distance magic labeling for each of them is given. In view of the fact that these are the only two connected tetravalent distance magic graphs up to order 15 which are not wreath graphs it makes sense to introduce the following family of graphs and investigate which of its members are distance magic. 
\begin{figure}[!h]
\centering
\includegraphics[scale=0.5]{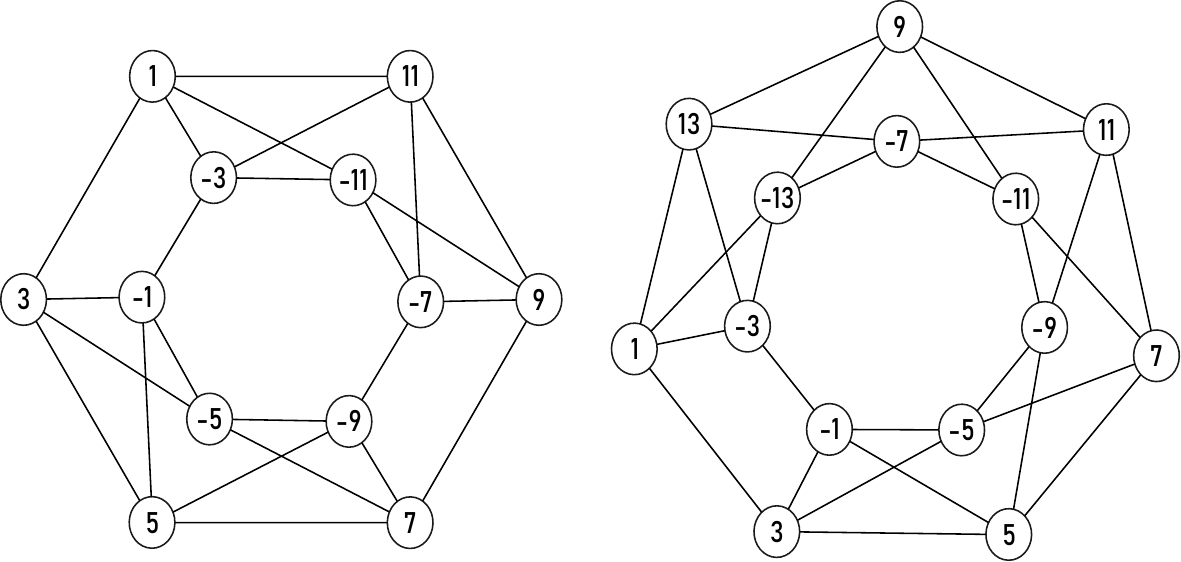}
\caption{The quasi wreath graphs $QW([0,1,1,0,1,1])$ and $QW([0,1,1,1,1,1,1])$.}
\label{pic:smallgraphs}
\end{figure}
\begin{construction} \label{constr:1}
Let $m \ge$ 3 be a positive integer and let $S= [s_0, s_1, \ldots ,s_{m-1}]$ be a sequence such that $s_i \in \{0,1\}$ for all $i \in \mathbb{Z}_m$, $s_0=0$, $s_{m-1}=1$ and for all $i \in \mathbb{Z}_m \setminus \{m-1\}$, at least one of $s_i$, $s_{i+1}$ is equal to 1. Then the {\em quasi wreath graph} $QW(S)$ is the tetravalent graph of order $2m$ with vertex set $\{x_i \colon i \in \mathbb{Z}_m\} \cup \{y_i \colon i \in \mathbb{Z}_m\}$ in which the edges are defined as follows:

\begin{itemize}[itemsep=1pt]
\item for each $i \in \ZZ_m$ we have the edges $\{x_i, x_{i+1}\}$ and $\{y_i, y_{i+1}\}$,
\item for each $i \in \ZZ_m$  with $s_i=0$ we have the edges $\{x_i, y_i\}$ and $\{x_{i+1}, y_{i+1}\}$ and
\item for each $i \in \ZZ_m$  with $s_i=1$ we have the edges $\{x_i, y_{i+1}\}$ and $\{x_{i+1}, y_i\}$,
\end{itemize}
where all the subscripts are computed modulo $m$.
\end{construction}

Note that the wreath graph $W(3)$ is isomorphic to the quasi wreath graph $QW([0,1,1])$ and that the graphs from Figure \ref{pic:smallgraphs} are isomorphic to the quasi wreath graphs $QW([0,1,1,0,1,1])$ and $QW([0,1,1,1,1,1,1])$. In the rest of this paper we abbreviate the term quasi wreath graph to {\em QW-graph}. We also simplify the notation and write $QW(a_1, a_2, \ldots, a_r)$ instead of $QW([\underbrace{0, 1, \ldots, 1}_{a_1}, \underbrace{0, 1, \ldots, 1}_{a_2}, \ldots, \underbrace{0, 1, \ldots, 1}_{a_r}])$. Note that with this notation, $m=a_1 + a_2 + \ldots + a_r$ and the above three graphs are $QW(3)$, $QW(3,3)$ and $QW(7)$, respectively.

As already mentioned, we determine all distance magic QW-graphs in this paper. Before starting this analysis we introduce some additional notation. Let $\Gamma=QW(S)$ be a QW-graph of order $2m$. For each $i \in \mathbb{Z}_m$ we let $B_i=\{x_i, y_i\}$. The sets $B_i$, $i \in \ZZ_m$, are called {\em blocks} of $\Gamma$. Suppose $i \in \ZZ_m$ is such that $s_i=0$ and let $j \ge 1$ be the smallest integer such that $s_{i+j}=0$ (the index computed modulo $m$). Then the subgraph of $\Gamma$ induced on $B_{i+1} \cup B_{i+2} \cup \cdots \cup B_{i+j}$ is called a {\em segment} of $\Gamma$ and is said to be of {\em length} $j$. Note that the number of segments of $\Gamma$ equals the number of $i \in \ZZ_m$ such that $s_i=0$ and that a QW-graph $QW(a_1, a_2, \ldots, a_r)$ has $r$ segments of lengths $a_1, a_2, \ldots, a_r$. 

Finally, for a $QW$-graph $\Gamma= QW(S)$ of order $2m$ and for a labeling $\ell$ of its vertices (not necessarily distance magic) we let $\ell_i=\ell(x_i) + \ell(y_i)$ for each $i \in \mathbb{Z}_m$ and  call it the {\em block label} of $B_i$. We also make an agreement that the indices of $s_i$, $x_i$, $y_i$, $\ell_i$ and $B_i$ are always computed modulo $m$.

\section{A necessary condition} \label{section:}
In this section we give a necessary condition for a QW-graph to be distance magic. First we point out two properties of distance magic labelings of QW-graphs that will be useful later on.

\begin{lemma} \label{lemma:1}
Let $QW(S)$ be a distance magic QW-graph and let $\ell$ be a corresponding distance magic labeling. Then for each $i \in \mathbb{Z}_m$ such that $s_{i}=0$ we have that $\ell_{i+1} \ne -\ell_{i+2}$. 
\end{lemma}

\begin{proof}
By way of contradiction, suppose that for some $i \in \ZZ_m$ we have that $s_{i}=0$ and $\ell_{i+1}=-\ell_{i+2}$, that is, $\ell(x_{i+1}) + \ell(y_{i+1})+\ell_{i+2}=0$. Considering the neighbors of $x_{i+1}$ we have that $\ell(x_i) + \ell(y_{i+1})+\ell_{i+2}=0$. This implies that $\ell(x_i)=\ell(x_{i+1})$, a contradiction. 
\end{proof}

\begin{lemma} \label{lemma:f}
Let $QW(S)$ be a distance magic QW-graph and let $\ell$ be a corresponding distance magic labeling. Then for each $i \in \ZZ_m$ the following holds:
\begin{equation}
\begin{split}
\ell_{i+2}
& = \left\{ 
\begin{array}{lll}
-\ell_i \colon & s_i=s_{i+1}=1\\
-\frac{1}{2}(\ell_i+\ell_{i+1}) \colon & s_i=0, \quad s_{i+1}=1 \\
-2\ell_i-\ell_{i+1} \colon & s_i= 1, \quad s_{i+1}=0.
\end{array} \right. \\
\end{split}
\end{equation}
\end{lemma}

\begin{proof}
The case of $s_i=s_{i+1}=1$ is clear (simply considering the neighbors of $x_{i+1}$). Suppose next that $s_i=0$ and $s_{i+1}=1$. Considering the neighbors of $x_{i+1}$ and $y_{i+1}$ we have that $\ell(x_i)+\ell(y_{i+1})+\ell_{i+2}=0=\ell(y_i)+\ell(x_{i+1})+\ell_{i+2}$. This yields $\ell_{i}+\ell_{i+1}+2\ell_{i+2}=0$, and therefore $\ell_{i+2}= -\frac{1}{2}(\ell_i+\ell_{i+1})$. The case of $s_i=1$, $s_{i+1}=0$ is analogous and is left to the reader.
\end{proof}

\begin{proposition} \label{proposition:necessary}
Let $\Gamma=QW(S)$ be a distance magic QW-graph and let $\ell$ be a corresponding distance magic labeling. Then all of the segments of $\Gamma$ are of odd length. Moreover, the number of segments whose length is congruent to 1 modulo 4 is even. 
\end{proposition}

\begin{proof}
Let $k \in \ZZ_m$ be such that $|\ell_k| \ge |\ell_t|$ for all $t \in \ZZ_m$. Note that Lemma \ref{lemma:1} implies that $\ell_k \ne 0$. Let $i,j$ be the smallest integers such that $i \ge 1$, $j \ge 0$ and $s_{k-i}=s_{k+j}=0$. Therefore, the subgraph induced on $B_{k-i+1} \cup B_{k-i+2} \cup \cdots \cup B_{k+j}$ is a segment (of length $i+j$) containing the block $B_k$ having the maximum absolute value of a block label among all blocks of $\Gamma$. Let $a = \ell_{k-i+1}$ and $b = \ell_{k-i+2}$. Since $s_{k-i+r}=1$ for all $r$ with $1 \le r < i+j$, Lemma \ref{lemma:f} implies that $\ell_{k-i} = -a-2b$ and that
\begin{equation} \label{eq:r}
\begin{split}
\ell_{k-i+r}
& = \left\{ 
\renewcommand{\arraystretch}{1.2}
\begin{array}{lll}
-b \colon & \modulo{r}{0}{4}\\
a \colon & \modulo{r}{1}{4}\\
b \colon & \modulo{r}{2}{4}\\
-a \colon & \modulo{r}{3}{4}
\end{array} \right. \\
\end{split}
\end{equation}
for all $r$ with $1 \le r \le i+j$ (see Figure \ref{fig:proof}). Moreover,
\begin{equation} \label{eq:first}
\begin{split}
\ell_{k+j+1}
& = \left\{ 
\renewcommand{\arraystretch}{1.2}
\begin{array}{lll}
2a+b \colon & \modulo{i+j}{0}{4}\\
2b-a \colon & \modulo{i+j}{1}{4} \\
-2a-b \colon & \modulo{i+j}{2}{4} \\
-2b+a \colon & \modulo{i+j}{3}{4}. 
\end{array} \right. \\
\end{split}
\end{equation}
\begin{figure}[!h] 
\centering
\includegraphics[scale=0.99]{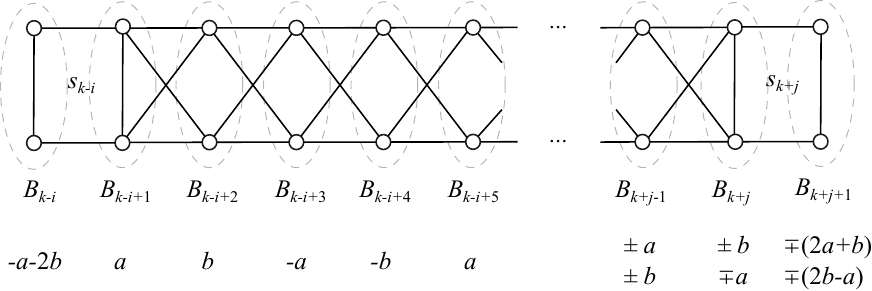}
\caption{Block labels of $QW(S)$.}
\label{fig:proof}
\end{figure}

\noindent
Suppose $i+j$ is even. Then (\ref{eq:first}) implies that $|\ell_{k+j+1}|=|2a+b|$, and so 
$$ \max \left\{|a|,|b|\right\} = |\ell_k| \ge \max\left\{ |\ell_{k-i}|, |\ell_{k+j+1}|\right\} = \max\left\{|a+2b|,|2a+b|\right\}$$
implies that $b=-a$, contradicting Lemma \ref{lemma:1}. Therefore, $i+j$ is odd, and so (\ref{eq:first}) implies that $|\ell_{k+j+1}|= |2b-a|$. It follows that
\begin{equation*}
\max \left\{|a|,|b|\right\} = |\ell_k| \ge \max\left\{ |\ell_{k-i}|, |\ell_{k+j+1}|\right\} = \max\left\{|2b+a|,|2b-a|\right\}.
\end{equation*}
Therefore $|a|>|b|$, and consequently $b=0$. Moreover, (\ref{eq:r}) and (\ref{eq:first}) imply that if $\modulo{i+j}{1}{4}$ then $\ell_{k+j}=a$ and $\ell_{k+j+1}=-a$, while if $\modulo{i+j}{3}{4}$ then $\ell_{k+j}=-a$ and $\ell_{k+j+1}=a$. Note that this implies that the sum of labels of all vertices of this segment equals 0 if $\modulo{i+j}{3}{4}$ and equals $a$ otherwise. Moreover, (\ref{eq:first}) implies that the next segment also contains a block having the maximum absolute value of a block label among all blocks of $\Gamma$, and so an inductive approach shows that all segments are of odd length, as claimed. Since $\Gamma$ is distance magic, the sum of labels of all vertices of the graph is $0$, showing that the number of segments whose length is congruent to 1 modulo 4 is even.
\end{proof}

\section{A sufficient condition}
In this section we show that the condition from Proposition \ref{proposition:necessary} is in fact sufficient for a QW-graph to be distance magic. We first introduce the following two terms. Let $\Gamma=QW(S)$ be a QW-graph. A segment of $\Gamma$ is said to be of {\em type A} if its length is congruent to 3 modulo 4, and is said to be of {\em type B} if its length is congruent to 1 modulo 4. The next result shows that if $\Gamma$ satisfies the condition of Proposition \ref{proposition:sufficient}, then it is distance magic. This is done by exhibiting a particular distance magic labeling for such graphs. To simplify the notation when working with pairs of integers we make the convention of writing the pair  $\left<-x,-y\right>$ as $(-1)\left<x,y\right>$ and the pair $\left<x+x', y+y'\right>$ as $\left<x,y \right> + \left<x',y'\right>$.

\begin{proposition} \label{proposition:sufficient}
Let $\Gamma=QW(S)$ be a QW-graph. If $\Gamma$ only has segments of types A and B and has an even number of segments of type B, then $\Gamma$ is distance magic.
\end{proposition}

\begin{proof}
Suppose all segments of $\Gamma$ are of types A and B and there is an even number of segments of type B.  Let $m = |S|$ and let $t$ be the number of segments of $\Gamma$. By definition there exist integers $k_1, k_2, \ldots, k_t$, where $0=k_1<k_2<k_3<\cdots <k_t<m-1$, such that for each $j$ with $0 \le j < m$, we have that $s_j=0$ if and only if $j \in \{k_1, k_2, \ldots, k_t\}$. \\
For each $i$ with $1 \le i \le t$ we call the segment containing all the blocks $B_j$ with $k_i+1 \le j \le k_{i+1}$ the {\em i-th segment} of $\Gamma$ (with the understanding that $k_{t+1}=m$) and we let $b_i$ be the number of segments of type B preceding the $i$-th segment where by definition $b_1=0$. Furthermore, we let $\Gamma_i$ be the subgraph of $\Gamma$ induced on $B_{k_i}\cup B_{k_i+1} \cup \cdots \cup B_{k_{i+1}-1}$ (see Figure \ref{fig:proof1}). 

We now assign labels to the vertices of $\Gamma$ as follows (see Figure \ref{pic:labeling} for an example of such a labeling). For each $i$ with $1 \le i \le t$ we first let

\begin{equation} \label{eq:j=0}
\langle \ell(x_{k_i}), \ell(y_{k_i}) \rangle = (-1)^{b_i}(\langle 2k_i, -2k_i \rangle + \langle1, -3 \rangle),
\end{equation}
\begin{equation} \label{eq:j=1}
\langle \ell(x_{k_i+ 1}), \ell(y_{k_i+ 1}) \rangle = (-1)^{b_i} (\langle 2k_i, -2k_i \rangle + \langle 3, -1\rangle),
\end{equation} 
and 
\begin{equation} \label{eq:j} 
\langle \ell(x_{k_i+ j}), \ell(y_{k_i+ j}) \rangle = (-1)^{b_i} (\langle 2k_i, -2k_i \rangle + \langle \alpha, -\beta \rangle),
\end{equation}
for each $j$ with $2 \le j \le k_{i+1} - k_i - 3$, where
\begin{equation*} \label{eq:6a}
\begin{split}
\left< \alpha, \beta \right> & =\left\{
\renewcommand{\arraystretch}{1.2}
\begin{array}{lll}
\langle 2j+3,\; 2j+3 \rangle \colon & \modulo{j}{0}{4}\\
\langle 2j-1,\; 2j-3 \rangle \colon & \modulo{j}{1}{4} \\
\langle 2j+1,\; 2j+1 \rangle \colon & \modulo{j}{2}{4}\\
\langle 2j+1,\; 2j+3 \rangle \colon & \modulo{j}{3}{4}.
\end{array} \right.
\end{split} \tag{6*}
\end{equation*}
In addition, for the $i$ with $1 \le i \le t$ for which the $i$-th segment is of type B and $b_i$ is even, we let $i'>i$ be the smallest positive integer such that the $i'$-th segment is also of type B, and we let
\begin{equation} \label{eq:7}
\langle \ell(x_{k_{i'+1}-2}), \ell(y_{k_{i'+1}-2}) \rangle =  \left<2k_{i+1}, -2k_{i+1}\right> + \left< -1, 3\right>,
\end{equation}
\begin{equation} \label{eq:8}
\langle \ell(x_{k_{i+1}-2}), \ell(y_{k_{i+1}-2}) \rangle =  \left<2k_{i+1}, -2k_{i+1}\right> + \left< -3, 1\right>, 
\end{equation} and
\begin{equation} \label{eq:9}
\langle \ell(x_{k_{i+1}-1}), \ell(y_{k_{i+1}-1}) \rangle =\left< 2k_{i'+1}, -2k_{i'+1} \right> + \left< -3, 3 \right>.
\end{equation}
Note that with (\ref{eq:7}), (\ref{eq:8}) and (\ref{eq:9}) we labeled four vertices of the $i$-th segment and two vertices of the $i'$-th segment. \\
Next, for the $i$ with $1 \le i \le t$ for which the $i$-th segment is of type A and of length greater than three we let
\begin{equation} \label{eq:10}
\langle \ell(x_{k_{i+1}-2}), \ell(y_{k_{i+1}-2}) \rangle = (-1)^{b_i} (\left< 2k_{i+1}, -2k_{i+1}  \right> + \left< -5, 7 \right>).
\end{equation}
Finally, for the $i$ with $1 \le i \le t$ for which the $i$-th segment is of type B with $b_i$ odd or of type A we let
\begin{equation} \label{eq:11}
\langle \ell(x_{k_{i+1}-1}), \ell(y_{k_{i+1}-1}) \rangle = \left< 2k_{i+1}, -2k_{i+1} \right> + \left< -1, 1 \right>.
\end{equation} 

\begin{figure}[!h] 
\centering
\includegraphics[scale=0.99]{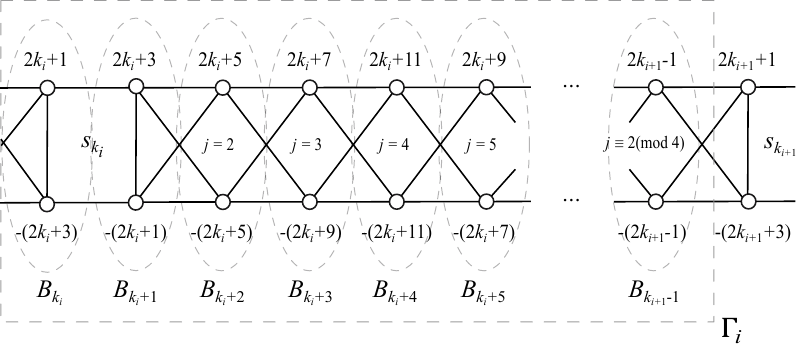}
\caption{Labeling of segments of type A (of length greater than 3) preceded by an even number of segments of type B.}
\label{fig:proof1}
\end{figure}

In order to prove that $\ell$ is a distance magic labeling of $\Gamma$ we proceed by proving two claims. \\

\noindent
\textsc{Claim 1}: The weight of each vertex of $\Gamma$ is equal to 0.

\noindent
Let $i \in \{1, \ldots, t\}$. Considering the labels of vertices of any subgraph $\Gamma_i$, for each $j$ with $2 \le j \le k_{i+1}-k_i-1$ we have that $N(x_{k_i+j})=N(y_{k_i+j})$. To see that these two vertices have their weight equal to 0 it is thus sufficient to consider the block labels $\ell_{k_i+j-1}$ and $\ell_{k_i+j+1}$. Note that the block label of each block in $\Gamma_i$ is $0$, $-2,$ or $2$. Namely, (\ref{eq:j=0}) and (\ref{eq:j=1}) imply that 
\begin{equation} \label{eq:asdfg}
\ell_{k_i}=-2(-1)^{b_i} \quad \textrm{and} \quad \ell_{k_{i+1}}=2(-1)^{b_i},
\end{equation}
and for each $j$ with $2 \le j \le k_{i+1}-k_i-3$, we have by (\ref{eq:j}) and (\ref{eq:6a}) that
\begin{equation} \label{eq:12}
\begin{split}
\ell_{k_i+j} & = \left\{
\renewcommand{\arraystretch}{1.2}
\begin{array}{lll}
0 \colon & \modulo{j}{0}{2}\\
2(-1)^{b_i}\colon & \modulo{j}{1}{4} \\
-2(-1)^{b_i} \colon & \modulo{j}{3}{4}.
\end{array} \right. \\
\end{split}
\end{equation}
Additionally, (\ref{eq:9}) and (\ref{eq:11}) imply that 
\begin{equation} \label{eq:added}
\ell_{k_{i+1}-1}=0
\end{equation}
and finally for segments of length greater than three (\ref{eq:7}), (\ref{eq:8}) and (\ref{eq:10}) imply that
\begin{equation} \label{eq:13}
\begin{split}
\ell_{k_{i+1}-2} & = \left\{
\renewcommand{\arraystretch}{1.2}
\begin{array}{lll}
2 (-1)^{b_i} \colon & \textrm{the } i\textrm{-th segment is of type A}\\
-2 (-1)^{b_i}\colon & \textrm{the } i\textrm{-th segment is of type B}.
\end{array} \right.
\end{split}
\end{equation}
We now verify that for each $j$ with $0 \le j \le k_{i+1}-k_i-1$ we indeed have that $w(x_{k_i+j})=w(y_{k_i+j}) = 0$. For $j$ with $3 \le j \le k_{i+1}-k_i-4$ this follows from (\ref{eq:12}). \\
Suppose that $j \in \{0,1\}$. By (\ref{eq:j=0}) and (\ref{eq:j=1}) we get that $\ell(x_{k_i})+\ell(y_{k_i+1})=0=\ell(y_{k_i})+\ell(x_{k_i+1})$. The first two vertices belong to $N(x_{k_i+1})$ and $N(y_{k_i})$, while the second two belong to $N(x_{k_i})$ and $N(y_{k_i+1}).$ Moreover, by (\ref{eq:12}) and (\ref{eq:added}) we have that $\ell_{k_i-1}=\ell_{k_i+2}=0$ which means that the vertices $x_{k_i}, y_{k_i}, x_{k_i+1}$ and $y_{k_i+1}$ all have weight 0. \\
Suppose that the $i$-th segment is of length three. Then $b_{i+1}=b_i$ (unless $i=t$ in which case $b_i$ is even), and so (\ref{eq:asdfg}) implies that $w(x_{k_i+2})=w(y_{k_i+2}) = 0$ which in turn implies that all vertices of $\Gamma_i$ have weight 0. From now on we thus assume that the $i$-th segment is of length at least 5, that is $k_{i+1}-k_i \ge 5$  (where we set $k_{t+1}=m$ if $i=t$). \\
For $j=k_{i+1}-k_i-3$ we consider two cases. If the $i$-th segment is of type A, then $\modulo{j}{0}{4}$ and $j \ge 4$. Hence, (\ref{eq:12}) and (\ref{eq:13}) imply that $\ell_{k_i+j-1}= -2(-1)^{b_i}$ and $\ell_{k_i+j+1}=2(-1)^{b_i}$. Otherwise, $\modulo{j}{2}{4}$  (recall that the segments are of odd length) and by (\ref{eq:13}), $\ell_{k_i+j+1}=-2(-1)^{b_i}$. Moreover, by (\ref{eq:asdfg}) or (\ref{eq:12}) (depending on whether the $i$-th segment is of length five or more, respectively) we have that $\ell_{k_i+j-1}= 2(-1)^{b_i}$. In both cases, the corresponding sum is 0. \\
To complete the case $j=2$ it suffices to consider the segments of length greater then five. In that case (\ref{eq:asdfg}) and (\ref{eq:12}) imply that  $w(x_{k_i+2})=w(y_{k_i+2}) = 0$. \\
Similarly, for $j=k_{i+1}-k_i-2$, (\ref{eq:12}) and (\ref{eq:added}) imply that $\ell_{k_i+j-1}=\ell_{k_i+j+1}=0$. \\
Finally, let  $j=k_{i+1}-k_i-1$. If the $i$-th segment is of type A, then $b_{i+1}=b_i$ (unless $i=t$ in which case $b_i$ is even), and so (\ref{eq:asdfg}) implies that $\ell_{k_{i+1}}= -2(-1)^{b_i}$, while (\ref{eq:13}) implies that $\ell_{k_{i+1}-2}=2(-1)^{b_i}$. On the other hand, if the $i$-th segment is of type B, then $b_{i+1}=b_i+1$ (unless $i=t$ in which case $b_i$ is odd), and so (\ref{eq:asdfg}) implies that $\ell_{k_{i+1}}= 2(-1)^{b_i}$, while (\ref{eq:13}) implies that $\ell_{k_{i+1}-2}=-2(-1)^{b_i}$. In both cases, the corresponding sum is 0. This finally proves Claim 1.\\

\begin{figure}[!h]
\centering
\includegraphics[scale=0.65]{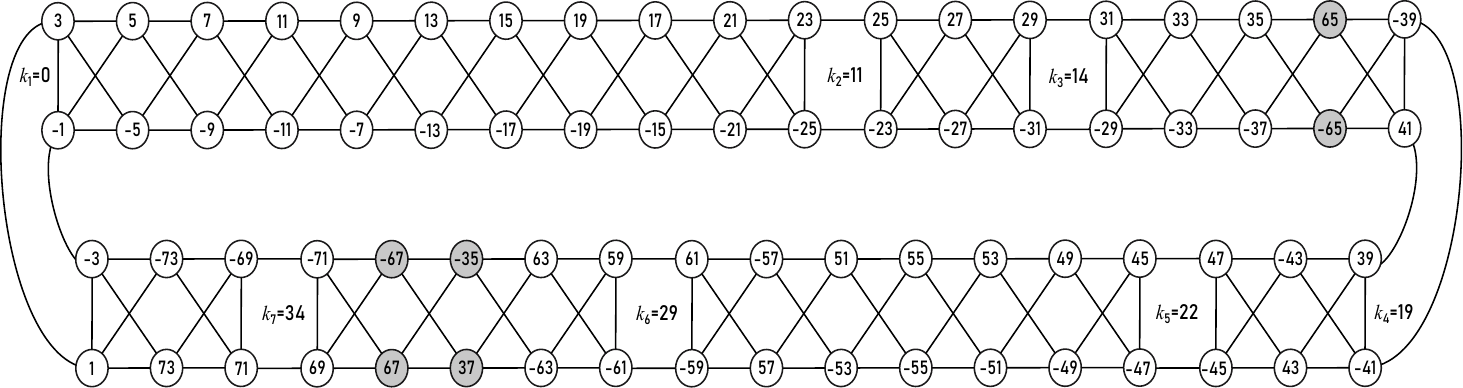}
\caption{A distance magic labeling of $QW(11,3,5,3,7,5,3)$ $(m=37)$ consisting of 5 segments of type A and 2 segments of type B, as defined in (\ref{eq:j=0})-(\ref{eq:11}).}
\label{pic:labeling}
\end{figure}

We now prove that the described labeling $\ell$ is a bijective mapping from $V(\Gamma)$ to $\{1-n, 3-n, \ldots, n-1\}$, where $n=2m$. 
It is clear that we lose nothing by exchanging a few labels from one segment with a few labels from another segment. To simplify this part of the proof, we do just that by letting $\tilde{\ell}$ be the labeling obtained from $\ell$ as follows. For each $i$ with $1 \le i \le t$ such that the $i$-th segment is of type B with $b_i$ even, let $i'>i$ be the smallest positive integer such that the $i'$-th segment is also of type B. Then $\tilde{\ell}$ on these two segments is defined precisely as $\ell$, except on the blocks $B_{k_{i+1}-1}$ and $B_{k_{i'+1}-2}$, for which we exchange the labels. In particular, instead of (\ref{eq:7}) and (\ref{eq:9}) we let
\begin{equation} \label{eq:7a}
\langle \tilde{\ell}(x_{k_{i'+1}-2}), \tilde{\ell}(y_{k_{i'+1}-2}) \rangle = \left< 2k_{i'+1}, -2k_{i'+1} \right> + \left< -3, 3 \right>, \tag{7*}
\end{equation}
\begin{equation} \label{eq:9a}
\langle \tilde{\ell}(x_{k_{i+1}-1}), \tilde{\ell}(y_{k_{i+1}-1}) \rangle = \left<2k_{i+1}, -2k_{i+1}\right> + \left< -1, 3\right>. \tag{9*}
\end{equation}

\noindent
\textsc{Claim 2}: For each $i \in \{1, \ldots, t\}$ the labeling $\tilde{\ell}$ maps $V(\Gamma_i)$ bijectively to $$\{1-2k_{i+1}, 3-2k_{i+1},\ldots, -2k_i-1\} \cup \{2k_i+1, 2k_i+3, \ldots, 2k_{i+1}-1\}.$$
Note that by definition of $\tilde{\ell}$ it is clear that for each $i$ the smallest absolute value of labels of vertices of $\Gamma_i$ is $2k_i+1$ and belongs to the vertices $x_{k_i}$ and $y_{k_i+1}$. \\
Suppose first that the $i$-th segment is of length 3. Then (\ref{eq:j=0}), (\ref{eq:j=1}) and (\ref{eq:11}) imply that Claim 2 clearly holds for $\Gamma_i$. Suppose now that the $i$-th segment is of length at least 5.
Computation of $\left< \alpha, \beta \right>$ in (\ref{eq:6a}) modulo 8 yields
\begin{equation*} \label{eq:16}
\begin{split}
\left< \alpha, \beta \right>
& \equiv\left\{ 
\renewcommand{\arraystretch}{1.2}
\begin{array}{lll}
\left< \modul{3}{8}, \; \modul{3}{8} \right> \colon & \modulo{j}{0}{4}\\
\left< \modul{1}{8}, \; \modul{7}{8} \right> \colon & \modulo{j}{1}{4}\\
\left< \modul{5}{8}, \; \modul{5}{8} \right> \colon & \modulo{j}{2}{4}\\
\left< \modul{7}{8}, \; \modul{1}{8} \right> \colon & \modulo{j}{3}{4}, 
\end{array} \right. \\
\end{split}
\end{equation*}
where $2\le j \le k_{i+1} - k_i - 3$. Thus, the labels of the vertices from each of the sets $\{x_{k_i+j} \colon 2 \le j \le k_{i+1} - k_i - 3\}$ and $\{y_{k_i+j} \colon 2 \le j \le k_{i+1} - k_i - 3\}$  are pairwise distinct. Moreover, since the labels of the vertices from one set are positive while those of the other are negative, all of these labels are pairwise distinct. Furthermore, since we require in (\ref{eq:j}) that $j \ge 2$, we have by (\ref{eq:6a}) that $\alpha, \beta \ge 5$, and so for each $i$ the absolute value of each label from (\ref{eq:j}) is greater than the absolute value of any label from (\ref{eq:j=0}) and (\ref{eq:j=1}). Hence, since $k_{i+1}-k_i$ is odd, (\ref{eq:6a}) implies that for all $j$ with $0 \le j \le k_{i+1}-k_i-4$ we have that $|\tilde{\ell}(x_{k_{i+1}-3})|, |\tilde{\ell}(y_{k_{i+1}-3})| > |\tilde{\ell}(x_{k_i+j})|, |\tilde{\ell}(y_{k_i+j})|$. \\
We now compare the labels of the vertices of $B_{k_{i+1}-2}$ and $B_{k_{i+1}-1}$ with the labels of the other vertices of $\Gamma_i$. Note that $k_{i+1}-k_i-3$ is congruent to 0 or 2 modulo 4, depending on whether the $i$-th segment is of type A or B, respectively. Thus (\ref{eq:j}) implies that
\begin{equation} \label{eq:s1}
\langle \tilde{\ell}(x_{k_{i+1}-3}), \tilde{\ell}(y_{k_{i+1}-3}) \rangle = (-1)^{b_i}(\left< 2k_{i+1}-3, -2k_{i+1}+3 \right>), \textrm{ or}
\end{equation}
\begin{equation} \label{eq:s2}
\begin{split}
\langle  \tilde{\ell}(x_{k_{i+1}-3}), \tilde{\ell}(y_{k_{i+1}-3}) \rangle = (-1)^{b_i}(\left< 2k_{i+1}-5, -2k_{i+1}+5 \right>), 
\end{split}
\end{equation}
depending on whether the $i$-th segment is of type A or B, respectively.
Suppose first, that the $i$-th segment is of type A. Then (\ref{eq:j}) implies that 
\begin{equation} \label{eq:s3}
\begin{split}
\langle \tilde{\ell}(x_{k_{i+1}-4}), \tilde{\ell}(y_{k_{i+1}-4}) \rangle & = (-1)^{b_i}(\left< 2k_{i+1}-7, -2k_{i+1}+5 \right>),
\end{split}
\end{equation}
and that $|\tilde{\ell}(x_{k_{i+1}-4})|, |\tilde{\ell}(y_{k_{i+1}-4})| > |\tilde{\ell}(x_{k_i+j})|, |\tilde{\ell}(y_{k_i+j})|$ for all $j$ with $0 \le j \le k_{i+1}-k_i-5$.
It follows from (\ref{eq:10}), (\ref{eq:11}), (\ref{eq:s1}) and (\ref{eq:s3}) that all the labels (according to $\tilde{\ell}$) of the vertices of $\Gamma_i$ are pairwise distinct and that the largest absolute value of a label in $\Gamma_i$ is $2k_{i+1}-1$.
We are left with the possibility that the $i$-th segment is of type B. Depending on whether $b_i$ is even or odd, (\ref{eq:8}), (\ref{eq:9a}) and (\ref{eq:s2}), or (\ref{eq:7a}), (\ref{eq:11}) and (\ref{eq:s2}), respectively, imply that all the labels of the vertices of $\Gamma_i$ are pairwise distinct and that the largest absolute value of a label in $\Gamma_i$ is $2k_{i+1}-1$.
Therefore, Claim 2 indeed holds for all $i$ with $1 \le i \le t$.

Note that Claim 2 implies that $\tilde{\ell}$ is a one to one correspondence from $V(\Gamma)$ to $\{1-2m, 3-2m, \ldots, 2m-1\}$. Combining this with Claim 1 we finally have that $\ell$ is a distance magic labeling of $\Gamma$.
\end{proof}

\section{Concluding remarks} \label{sec:future}

Recall that $QW(3)$ and $W(3)$ are isomorphic. However, for each $n \ge 4$ the wreath graph $W(n)$ has no 3-cycles, while each (distance magic) $QW$-graph of order $2n$ has 3-cycles (for instance, $(x_1, x_2, y_1)$). This shows that Theorem \ref{t1} provides an infinite family of tetravalent distance magic graphs which are not wreath graphs. In fact, it is easy to see that for each even order $n \ge 18$, there exists at least one distance magic $QW$-graph of order $n$. We mention that the above argument can also be used to verify that the $QW$-graphs of order at least 8 are not vertex-transitive (a graph is {\em vertex-transitive} if its automorphism group acts transitively on its vertex set), while all the wreath graphs of course are. The family of $QW$-graphs thus provides infinitely many tetravalent distance magic graphs which are not vertex-transitive.

A computer search reveals that out of 8\,037\,418 connected tetravalent graphs of order 16 only two are not ruled out by Corollary \ref{cor} as candidates for being distance magic. One of them is of course $W(8)$ and is thus distance magic. The other one turns out to also be distance magic but does not belong to the family of quasi wreath graphs (by Proposition \ref{proposition:necessary}). Intriguingly, this ``mysterious'' graph can be derived from $QW(7)$ using a construction from \cite{KFK}. 

The proof of \cite[Lemma 2.1]{KFK} shows that if $\Gamma$ is a tetravalent distance magic graph of order $n$ admitting a distance magic labeling (according to our definition) such that for some 4-cycle $C$ of $\Gamma$ the sum of the labels of each pair of antipodal vertices is 0, then one can construct a tetravalent distance magic graph $\Gamma'$ of order $n+2$ from $\Gamma$ by deleting the edges of $C$, adding two new vertices and joining them to each vertex of $C$. Applying this construction to the distance magic labeling of $QW(7)$ from Figure \ref{pic:smallgraphs} and the 4-cycle $(x_3, x_4, y_5, y_4)$ we thus obtain a tetravalent distance magic graph of order 16 which is clearly not a wreath graph. It thus must be the above mentioned ``mysterious''  distance magic graph of order 16. 

An exhaustive computer search reveals that there are 4 and 21 distance magic graphs of orders 17 and 18, respectively, which are neither wreath graphs nor quasi wreath graphs. It would therefore be interesting to further explore what kind of tetravalent distance magic graphs one can obtain by applying the construction from \cite{KFK} (and similar other constructions) to wreath and quasi wreath graphs.

\section*{Statements and Declarations}
Funding: \\
This work is supported in part by the Slovenian Research and Innovation Agency (Young researchers program, research program P1-0285 and research projects J1-2451 and J1-3001).
\medskip

\noindent
Competing interests: \\
The authors declare that they have no known competing financial or non-financial interests that are directly or indirectly related to this work.


\begin{thebibliography}{99}
\bibitem{ACP} M. Anholcer, S. Cichacz, I. Peterin, A. Tepeh, Distance magic labeling and two products of graphs, {\em Graphs Combin.} \textbf{31} (2015), 1125--1136.
\bibitem{CF} S. Cichacz, D. Froncek, Distance magic circulant graphs, {\em Discrete Math.} \textbf{339} (2016), 84--94.
\bibitem{Houseofgraphs} K. Coolsaet, S. D'hondt, J. Goedgebeur, House of Graphs 2.0: A database of interesting graphs and more, {\em Discrete Appl. Math.} \textbf{325} (2023), 97--107. Available at https://houseofgraphs.org. 
\bibitem{Gallian} J. A. Gallian, A dynamic survey of graph labeling, {\em Electron. J. Combin.} (2022) DS6. 
\bibitem{GodSin18} A.~Godinho, T.~Singh, Some distance magic graphs, {\em AKCE Int. J. Graphs Comb.} {\bf 15} (2018), 1--6.
\bibitem{KFK}  P. Ková\v{r}, D. Fronček, T. Ková\v{r}ová, A note on 4-regular distance magic graphs, {\em Australas. J. Combin.} \textbf{54} (2012), 127--132.
\bibitem{Meringer} M. Meringer, Fast generation of regular graphs and construction of cages, {\em J. Graph Theory} \textbf{30} (1999), 137--146.
\bibitem{MS} Š. Miklavič, P. Šparl, Classification of tetravalent distance magic circulant graphs, {\em Discrete Math.} \textbf{344} (2021), 112557. 
\bibitem{MS6} Š. Miklavič, P. Šparl, On distance magic circulants of valency 6, {\em Discrete Appl. Math.} \textbf{329} (2023), 35--48.
\bibitem{Miller}  M. Miller, C. Rodger, R. Simanjuntak, Distance magic labelings of graphs, {\em Australas. J. Combin.} \textbf{28} (2003), 305--315.
\bibitem{Rao} S. B. Rao, Sigma graphs - a survey, in: B. D. Acharya, S. Arumugam and A. Rosa (Eds.), {\em Labelings of Discrete Structures and Applications}, Narosa Publishing House, New Delhi (2008) 135--140.
\bibitem{RSP} S. B. Rao, T. Singh, V. Parameswaran, Some sigma labeled graphs: I, in: S. Arumugam, B. D. Acharya and S. B. Rao (Eds.), {\em Graphs, Combinatorics, Algorithms and Applications}, Narosa Publishing House, New Delhi (2004), 125--133. 
\bibitem{RS} K. Rozman, P. Šparl, Distance magic labelings of Cartesian products of cycles, in preparation. 
\bibitem{SFM} K. A. Sugeng, D. Froncek, M. Miller, J. Ryan, J. Walker, On distance magic labeling of graphs, {\em J. Combin. Math. Combin. Comput}. \textbf{71} (2009) 39--48. 
\bibitem{Vilfred} V. Vilfred, $\Sigma${\em -labelled graph and Circulant Graphs}, Ph.D. Thesis, University of Kerala, Trivandrum, India, 1994.
 
\end{thebibliography}
\end{document}